\documentclass{article}

\usepackage{amsmath}
\usepackage{amsthm}
\usepackage{amssymb}
\usepackage{graphicx}
\usepackage[inline]{enumitem}
\usepackage{hyperref}
\usepackage{xcolor}
\usepackage[round, authoryear, sort]{natbib}

\usepackage[a4paper, total={6in, 8in}]{geometry}

\usepackage{pgfplots}
\pgfplotsset{compat=1.18}
  
\newtheorem{theorem}{Theorem}
\newtheorem{lemma}{Lemma}

\newtheorem{claim}{Claim}

\newenvironment{proofclaim}[1][Proof of Claim]{%
  \pushQED{\qed}%
  \begin{proof}[#1]%
}{%
  \end{proof}%
}

\newenvironment{proofwithclaims}[1][Proof]{%
  \setcounter{claim}{0}%
  \begin{proof}[#1]%
}{%
  \end{proof}%
}

\newcommand{\E}{\mathbb{E}}
\newcommand{\R}{\mathbb{R}}
\renewcommand{\L}{\mathcal{L}}

\newcommand{\LAS}{\operatorname{LAS}}

\usepackage{todonotes}
\colorlet{greentransparent}{green!45}

\title{The Lasserre Rank of the Cropped Hypercube\vspace{-0em}}
\author{G\'erard Cornu\'ejols, Vrishabh Patil, Jiaye Wei}

\begin{document}

\maketitle

\begin{abstract}
     \noindent In an $n$-dimensional \emph{cropped hypercube} each of the $2^n$ cropping inequalities chops off a single corner of the $0$--$1$ hypercube by an $\ell_1$-distance $\rho$. The case $\rho = 1/2$ has been extensively studied in the literature.
     This paper shows that the Lasserre rank of the $n$-dimensional cropped hypercube where $\rho = 1/2$, $n \geq 2$, is the smallest integer $0\leq t \leq n$ such that $\Delta_t < 0$ in the recurrence $\Delta_{-1} = 1$, $\Delta_{0} = n-1$, $\Delta_t = (n-1)\Delta_{t-1} - t(n-t+1)\Delta_{t-2}$. It follows that the Lasserre rank can be computed in time $O(n^2 \log^2 n)$. Asymptotically, the rank is $\frac{n}{2} + c_{1/2}\sqrt{n} + o(\sqrt{n})$, where $c_{1/2}$ is the unique zero of a given function. Numerically, $c_{1/2} \approx 0.3825$. In fact, we prove such results for any fixed $0 < \rho < 1$.
\end{abstract}


\section{Introduction} \label{sec:intro}

Fix $0<\rho<1$ and let $N := \{1, \ldots, n\}$. The $n$-dimensional \emph{cropped hypercube} is the polytope
\begin{equation*}
     Q_{n,\rho} := \Bigl\{ x \in [0, 1]^n: g_I(x) \geq 0 \text{ for all } I \subseteq N \Bigr\},
\end{equation*}
where
\begin{equation} \label{eq:crop}
     g_I(x) := \sum_{i \in I} x_i + \sum_{i \in N \setminus I} (1 - x_i) - \rho.
\end{equation}
Geometrically, each vertex of the $0$-$1$ hypercube is cut off by a single inequality in the description of $Q_{n,\rho}$ at $\ell_1$-distance $\rho$. Together, the $2^n$ inequalities exclude all $2^n$ vertices.

Writing $S_{n,\rho} := Q_{n,\rho} \cap \{0, 1\}^n$, we have $S_{n,\rho} = \emptyset$ and therefore $\operatorname{conv}(S_{n,\rho}) = \emptyset$.

The cropped hypercube has been used as a canonical example for comparing cutting-plane and lift-and-project procedures.
In particular, the cropped hypercube obtained by taking $\rho=1/2$ has been studied extensively in the literature. 
The {\it rank} is the number of rounds of the procedure needed to reach $\operatorname{conv} (P \cap \mathbb{Z}^n)$  starting from a polyhedron $P$. 
The Chv\'atal--Gomory rank of $Q_{n,1/2}$ is $n$ (\cite{bockmayr1999chvatal}). \cite{laurent2003comparison} showed that the Sherali--Adams hierarchy also requires $n$ rounds for  $Q_{n,1/2}$ whereas the Lasserre rank is at most $n - 1$. For general polytopes, however, \cite{cheung2007computation} established that the Lasserre rank is $n$ in the worst case. Determining the precise Lasserre rank of the cropped hypercube has since become a natural test case for the strength of sum-of-squares relaxations. \cite{kurpisz2016tight} proved lower and upper bounds of $\Omega(\sqrt{n})$ and $n - \Omega(n^{1/3})$, respectively.  \cite{kurpisz2019boolean} sharpened the bounds for $Q_{n,1/2}$ as follows.
\begin{equation*}
    \left\lceil \frac{n}{2} \right\rceil \leq \operatorname{rank}_{\LAS}(Q_{n,1/2}) \leq \left\lceil \frac{n}{2} + \sqrt{n\log(2n)} \right\rceil.
\end{equation*}
More recently, \cite{kurpisz2023empty} established a lower bound of $n/2 + \Omega(\sqrt{n})$. Their work reduces the question of identifying the Lasserre rank of $Q_{n,1/2}$ to one quadratic inequality and then to positive-definiteness conditions for a tridiagonal matrix. These conditions yield lower and upper bounds on the rank but do not identify the exact transition for every $n$.
Symmetry reductions and univariate polynomial techniques have also been used more generally to analyze sum-of-squares certificates over the 0--1 hypercube (\cite{kurpisz2016sum,kurpisz2026certification}).

\cite{au2018elementary} studied the Lasserre rank of cropped hypercubes when $\rho$ decreases as $n$ increases. They show that the rank is $n$ when $\rho$ goes to 0 sufficiently fast.

In this paper, our three main contributions are:
\begin{enumerate}
    \item a simple recurrence for computing the Lasserre rank of $Q_{n,\rho}$. For every fixed rational $\rho$, the rank can be computed in $O(n^2 \log^2 n)$ time.  
    \item the following asymptotic formula for fixed $0 < \rho < 1$
    \begin{equation} \label{eq:asymt}
        \operatorname{rank}_{\LAS}(Q_{n,\rho}) = \frac{n}{2} + c_\rho\sqrt{n} + o(\sqrt{n}),
    \end{equation}
    where $c_\rho$ is the unique zero of the function $\int_0^\infty s^\rho e^{-s^2/8} \cos \left(cs + \frac{\pi\rho}{2}\right) \, ds$. For $\rho = 1/2$, we have $c_{1/2} \approx 0.3825$. This closes the previous $\sqrt{\log n}$ gap in the second-order term and determines its sharp constant.
    \item an exact characterization of the rank $n$ regime. There is a unique threshold $\rho_n \in (0,1)$ satisfying
    \begin{equation*}
        \rho_n \sum_{r=1}^n \frac{\binom{n}{r}}{r-\rho_n}=1,
    \end{equation*}
    such that $\operatorname{rank}{\LAS}(Q_{n,\rho}) = n$ if and only if $0 < \rho \leq \rho_n$. Moreover, $\rho_n \sim n/2^{n+1}$, closing the asymptotic factor-of-two gap between the previous bounds.
\end{enumerate}
These results are obtained through a sequence of reductions from an exponentially indexed semidefinite system to a weighted least-squares problem. At level $t$, the Lasserre relaxation is empty if and only if this least-squares problem with $t$ variables has optimum strictly below $\rho$. This gives a different route to the tridiagonal structure identified by \cite{kurpisz2023empty}. Our first new step is to prove that the sign of the least-squares optimum is exactly the sign of the determinant of an explicit tridiagonal matrix, which can be computed using a three-term recurrence. Relating this recurrence with Krawtchouk polynomials then allows us to determine the exact Lasserre rank for every $n$ and $0 < \rho < 1$.
The least-squares formulation can be solved directly for $n$ up to $250$. For example, when $\rho=1/2$, the ranks for $n = 10$, $100$, and $250$ are $6$, $53$, and $131$, respectively. The recurrence makes much larger computations routine. Still with $\rho=1/2$, for $n = 10^3$, $10^4$, and $10^5$, the corresponding ranks are $512$, $5038$, and $50120$, respectively.
It is worth noting that Krawtchouk polynomials were already used  by \cite{SlotLaurent} in the context of the Lasserre hierarchy.

The remainder of the paper is organized as follows. Section~\ref{sec:lasserre-hierarchy} recalls the Lasserre hierarchy. Section~\ref{sec:symmetry} uses the symmetries of the cropped hypercube to reduce feasibility to one canonical lifted point and one cropping inequality. Section~\ref{sec:least-squares} derives the least-squares characterization. Section~\ref{sec:rank-computation} derives the tridiagonal and Krawtchouk representations and analyzes the critical sign transition.
Section~\ref{sec:asymptotic-analysis} presents the asymptotic analysis.

\section{The Lasserre Hierarchy} \label{sec:lasserre-hierarchy}

We recall the Lasserre hierarchy for polyhedra described by linear inequalities (\cite{lasserre2001explicit,laurent2003comparison}). Let
\begin{equation*}
    P := \{x \in [0, 1]^n: g_\ell(x) \geq 0,\ \ell = 1, \ldots, m\}, \qquad S := P \cap \{0, 1\}^n,
\end{equation*}
where each $g_\ell$ is a linear function. The lifted variables are $y = (y_I)_{I \subseteq N}$, intended to represent the monomials $\prod_{i \in I} x_i$. Since $x_i^2 = x_i$ for $x_i \in \{0, 1\}$, the product of two monomials indexed by $I$ and $J$ is indexed by $I \cup J$.

For $r \geq 0$, the \emph{moment matrix} of order $r$ is indexed by the subsets of $N$ of cardinality at most $r$ and has entries
\begin{equation} \label{eq:moment-matrix}
    [M_r(y)]_{I,J} := y_{I \cup J}, \qquad |I|, |J| \leq r.
\end{equation}
If $g(x) = b - \sum_{i=1}^n a_i x_i$ is linear, define the shifted sequence $g \cdot y$ by
\begin{equation} \label{eq:shifted-sequence}
    (g \cdot y)_K := b y_K - \sum_{i=1}^n a_i y_{K \cup \{i\}},
    \qquad K \subseteq N.
\end{equation}
The matrix $M_r(g \cdot y)$ is the \emph{localizing matrix} of $g$ of order $r$. Its $(I,J)$-entry is
\begin{equation} \label{eq:localizing-matrix}
    [M_r(g \cdot y)]_{I,J} = b y_{I \cup J} - \sum_{i=1}^n a_i y_{I \cup J \cup \{i\}}.
\end{equation}

For an integer $t \geq 0$, the lifted level-$t$ relaxation is
\begin{equation} \label{eq:lasserre-lifted}
    \widetilde{\LAS}_t(P) := \left\{y: y_\emptyset = 1, \quad M_{t + 1}(y) \succeq 0, \quad M_t(g_\ell \cdot y) \succeq 0\ \text{for } \ell = 1, \ldots, m \right\},
\end{equation}
and its projection onto the space of original variables is
\begin{equation} \label{eq:lasserre-projected}
    \LAS_t(P) := \left\{x \in \R^n: x_i = y_{\{i\}} \text{ for all } i \in N \text{ and some } y \in \widetilde{\LAS}_t(P) \right\}.
\end{equation}
We use the convention in which the moment matrix has order $t + 1$ and the localizing matrices have order $t$. The bound inequalities $x_i \geq 0$ and $1 - x_i \geq 0$ need not be included among the $g_\ell$. Their localizing-matrix conditions are implied by $M_{t + 1}(y) \succeq 0$ (\cite{laurent2003comparison}, Lemma~5).

To see the meaning of these semidefinite conditions, fix $v \in \{0, 1\}^n$ and set $\mu(v)_I := \prod_{i \in I} v_i$ for $I \subseteq N$. Then $M_r(\mu(v))$ is the rank-one matrix whose $(I,J)$-entry is $\mu(v)_I \mu(v)_J$, and
\begin{equation*}
    M_r(g \cdot \mu(v)) = g(v)M_r(\mu(v)).
\end{equation*}
Consequently, $\mu(v)$ is feasible for \eqref{eq:lasserre-lifted} whenever $v \in S$, and hence $\operatorname{conv}(S) \subseteq \LAS_t(P)$. The hierarchy is nested and satisfies $\LAS_n(P) = \operatorname{conv}(S)$ (\cite{lasserre2001explicit,laurent2003comparison}). Its \emph{Lasserre rank} is the least $t$ for which $\LAS_t(P) = \operatorname{conv}(S)$.

We assume that $n \geq 2$ for the remainder of the paper.

Figure~\ref{fig:Q_3-lasserre} shows the three relaxations
$\LAS_t(Q_{3,1/2})$ for $t = 0, 1, 2$. In particular, $\LAS_0(Q_{3,1/2}) = Q_{3,1/2}$, the nonempty spectrahedron $\LAS_1(Q_{3,1/2})$ contains $\left(\tfrac{1}{2}, \tfrac{1}{2}, \tfrac{1}{2}\right)$, and $\LAS_2(Q_{3,1/2}) = \emptyset$. Of course, $\LAS_3(Q_{3,1/2}) = \emptyset$
as well.

\begin{figure}[htbp]
    \centering
    \includegraphics[width=\textwidth]{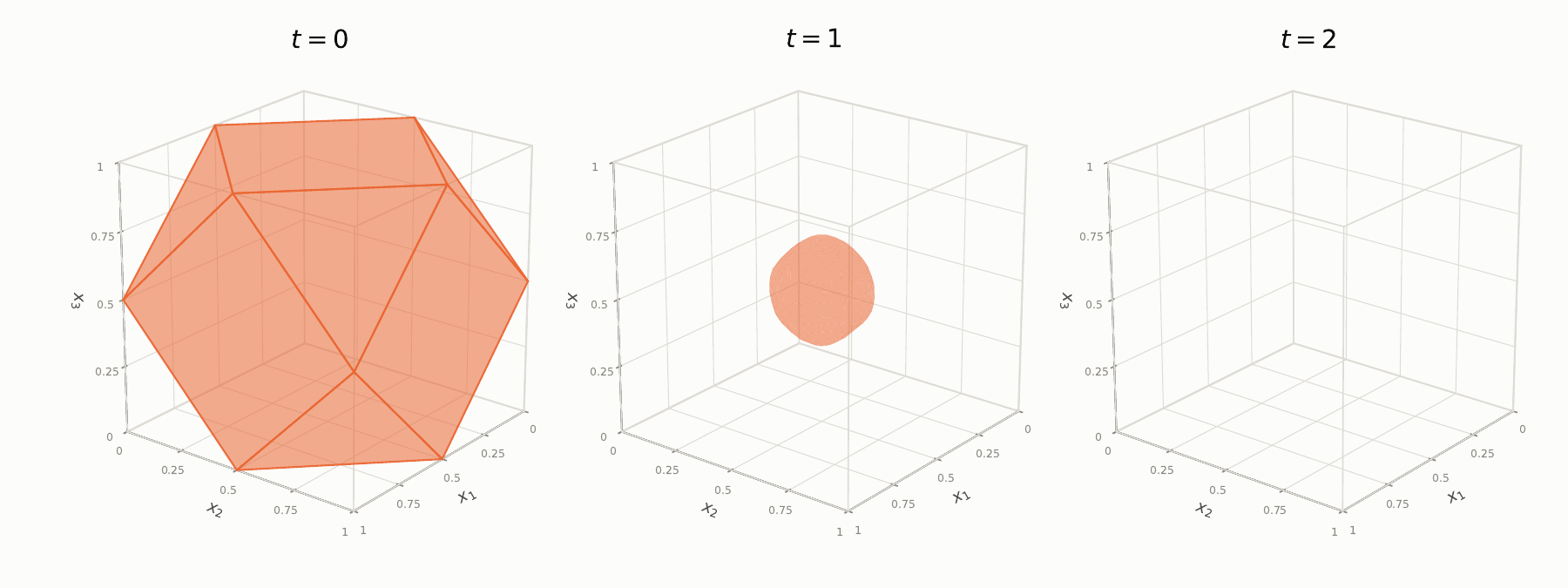}
    \caption{The Lasserre relaxations for $Q_{3,1/2}$ at levels $t = 0, 1, 2$.}
    \label{fig:Q_3-lasserre}
\end{figure}

Henceforth, we write $\widetilde{\LAS}_t$ and $\LAS_t$ for $\widetilde{\LAS}_t(Q_{n,\rho})$ and $\LAS_t(Q_{n,\rho})$, respectively. In the next section, we exploit the symmetries of $Q_{n,\rho}$ to decide this infeasibility question at a single canonical moment vector $y^*$.

\section{Symmetry Reduction} \label{sec:symmetry}

We exploit the symmetries of the cropped hypercube to reduce the feasibility of the level-$t$ Lasserre relaxation to a single lifted vector $y^*$ and, subsequently, to a single cropping inequality. This isolates the only localizing matrix that needs to be analyzed.

Let $\Gamma$ be the group generated by coordinate permutations and coordinate complementations $\tau_i : x_i \mapsto 1 - x_i$. A permutation $\sigma$ maps $g_I$ to $g_{\sigma(I)}$, while $\tau_i$ maps $g_I$ to $g_{I \triangle \{i\}}$. Both operations also permute the bound inequalities. Thus, every element of $\Gamma$ preserves the formulation.

We will use the following technical lemma to prove our main results. For an integer $k \geq 0$, let $m_k(x)$ denote the column vector of monomials $\prod_{i \in I} x_i$ with $I \subseteq N$ such that $|I| \leq k$.

\begin{lemma} \label{lemma:gamma-matrix}
    Order the entries of $m_k(x)$ first by degree and then in any fixed order within each degree. For every $\gamma \in \Gamma$, there is an invertible matrix $C_\gamma$ such that
    \begin{equation} \label{eq:gamma-matrix}
        m_k(\gamma(x)) = C_\gamma m_k(x).
    \end{equation}
    If $\gamma$ is a coordinate permutation, then $C_\gamma$ is a permutation matrix. If $\gamma$ is a complementation of a set of coordinates, then $C_\gamma$ is triangular with diagonal entries in $\{\pm 1\}$, and hence is invertible.
\end{lemma}

\begin{proof}
    A coordinate permutation simply permutes the monomials of each degree, so it induces a permutation matrix on the entries of $m_k(x)$.

    Now suppose that $\gamma$ complements the coordinates in a set $I \subseteq N$. For a monomial indexed by $J \subseteq N$ with $|J| \leq k$, we have
    \begin{equation*}
         \prod_{i \in J} \gamma(x)_i = \prod_{i \in J \cap I} (1 - x_i) \prod_{i \in J \setminus I} x_i = \sum_{T \subseteq J \cap I} (-1)^{|T|} \prod_{i \in (J \setminus I) \cup T} x_i.
    \end{equation*}

    Every monomial on the right has degree at most $|J|$. Therefore, with the degree ordering above, the matrix $C_\gamma$ is triangular. The coefficient of the original monomial $\prod_{i \in J} x_i$ is $(-1)^{|J \cap I|}$, which is either $1$ or $-1$. Thus all diagonal entries of $C_\gamma$ are $\pm 1$, so $C_\gamma$ is invertible because the determinant of a triangular matrix is the product of its diagonal entries.
\end{proof}

Define $y^*$ by
\begin{equation*}
     y^*_I := \left( \frac{1}{2} \right)^{|I|}, \qquad I \subseteq N.
\end{equation*}
Note that $y^* = \sum_{v \in \{ 0, 1 \}^n} \lambda_v \mu(v)$ with equal weights $\lambda_v = 2^{-n}$ on all $v$, where $\mu(v)_I := \prod_{i \in I} v_i$ (recall the definition from the previous section).
 
For $\rho=1/2$, Laurent showed that the vector $y^*$ is feasible at level $n - 1$ of the Sherali--Adams relaxation for the cropped hypercube (see \cite{laurent2003comparison}, Proposition~11). We prove the following.

\begin{lemma} \label{lemma:symmetry}
     The relaxation $\LAS_t$ is nonempty if and only if $y^*$ belongs to $\widetilde{\LAS}_t$.
\end{lemma}

\begin{proofwithclaims}
    One direction is straightforward. If $y^* \in \widetilde{\LAS}_t$, then its projection belongs to $\LAS_t$, which is therefore nonempty.

    For the converse, suppose that $\LAS_t$ is nonempty and let $y \in \widetilde{\LAS}_t$. Each $\gamma \in \Gamma$ induces an invertible linear map on the lifted variables by substituting $\gamma(x)$ into each monomial and replacing every resulting monomial indexed by $I$ with $y_I$. Indeed, $\gamma$ is a composition of a permutation and complementations: a permutation gives $(\sigma \cdot y)_I = y_{\sigma(I)}$, while a complementation gives
    \begin{equation} \label{eq:flip}
        (\tau_i \cdot y)_I =
        \begin{cases}
           y_{I \setminus \{i\}} - y_I, & i \in I, \\
           y_I, & i \notin I.
        \end{cases}
    \end{equation}
    Both maps fix $y_\emptyset = 1$.

    \medskip
    \begin{claim}
        The vector $\bar{y} = \frac{1}{|\Gamma|} \sum_{\gamma \in \Gamma} \gamma \cdot y$ belongs to $\widetilde{\LAS}_t$.
    \end{claim}
    \begin{proofclaim}
         Observe that the feasible region $\widetilde{\LAS}_t$ is convex, and that it is invariant under $\gamma \in \Gamma$. 
         
        Indeed, $\gamma$ permutes the $2^n$ inequalities $g_I$ among themselves. For the matrix conditions themselves, $M_k(y)$ is the linearization of $m_k(x)\,m_k(x)^\top$. By Lemma~\ref{lemma:gamma-matrix}, $m_k(\gamma(x)) = C_\gamma\,m_k(x)$ for an invertible matrix $C_\gamma$. Consequently,
        \begin{equation*}
             M_k(\gamma \cdot y) = C_\gamma M_k(y)C_\gamma^\top.
        \end{equation*}
        Moreover, if $J \subseteq N$ is determined by $g_I(\gamma(x)) = g_J(x)$, then
        \begin{equation*}
            M_k(g_I \cdot (\gamma \cdot y)) = C_\gamma M_k(g_J \cdot y)C_\gamma^\top.
        \end{equation*}
        For any positive semidefinite matrix $M$ and any invertible matrix $C$, the matrix $C M C^\top$ is also positive semidefinite because $v^\top C M C^\top v = (C^\top v)^\top M(C^\top v) \geq 0$ for every $v$. Taking $k = t+1$ for the moment matrix and $k=t$ for the localizing matrices therefore shows that $\gamma$ preserves all the semidefinite conditions in \eqref{eq:lasserre-lifted}.
    
        Hence $\gamma \cdot y$ is feasible for every $\gamma \in \Gamma$, and by convexity so is the average $\bar{y} = \frac{1}{|\Gamma|} \sum_{\gamma \in \Gamma} \gamma \cdot y.$
    \end{proofclaim}

    It remains to identify $\bar{y}$, which satisfies $\bar{y}_\emptyset = 1$ and $\gamma \cdot \bar{y} = \bar{y}$ for every $\gamma \in \Gamma$, since averaging over all of $\Gamma$ and then applying one more map reproduces the same average. Applying \eqref{eq:flip} with $i \in I$, that invariance forces $\bar{y}_I = \bar{y}_{I \setminus \{i\}} - \bar{y}_I$, that is $\bar{y}_I = \frac{1}{2} \bar{y}_{I \setminus \{i\}}$, and induction on $|I|$ from $\bar{y}_\emptyset = 1$ gives $\bar{y}_I = (\frac{1}{2})^{|I|}$. So $\bar{y} = y^*$, which is therefore feasible for $\widetilde{\LAS}_t$.
\end{proofwithclaims}

\bigskip
Consequently, the Lasserre rank of $Q_{n,\rho}$ is the smallest $t$ at which $y^*$ fails one of the level-$t$ semidefinite conditions. This reduces the question to evaluating explicit matrices at a single rational vector.

We next show that the moment matrix alone cannot distinguish the cropped hypercube from its integer hull. For a vector $\alpha$ indexed by the subsets $\{I \subseteq N: |I| \leq k\}$, define the polynomial
\begin{equation*}
    p_\alpha(x) := \alpha^\top m_k(x) = \sum_{I \subseteq N: |I| \leq k} \alpha_I \prod_{i \in I} x_i.
\end{equation*}

\begin{lemma} \label{lemma:moment-psd}
    $M_k(y^*) \succeq 0$ for every $k$ and every $n$. More generally, if $g \geq 0$ at every point of $\{0, 1\}^n$, then $M_k(g \cdot y^*) \succeq 0$ for every $k$.
\end{lemma}

\begin{proof}
    The definition of $y^*$ gives
    \begin{equation*}
        y^*_{I \cup J} = \sum_{v \in \{0, 1\}^n} 2^{-n} \prod_{i \in I \cup J} v_i.
    \end{equation*}
    Fix any coefficient vector $\alpha$ indexed by $\{I: |I| \leq k\}$. Since $v_i^2 = v_i$ for $v \in \{0, 1\}^n$, we have $\prod_{i \in I} v_i \prod_{j \in J} v_j = \prod_{i \in I \cup J} v_i$. It follows that
    \begin{equation*}
        \alpha^\top M_k(y^*)\, \alpha = \sum_{I,J} \alpha_I \alpha_J \sum_{v} 2^{-n} \prod_{i \in I \cup J} v_i = \sum_v 2^{-n} \, p_\alpha(v)^2 \geq 0.
    \end{equation*}
    For the second claim, observe that for every $K \subseteq N$,
    \begin{equation*}
        (g \cdot y^*)_K = \sum_{v \in \{0, 1\}^n} 2^{-n} g(v) \prod_{i \in K} v_i.
    \end{equation*}
    The identical computation gives
    \begin{equation*}
        \alpha^\top M_k(g \cdot y^*) \alpha = \sum_{v \in \{0, 1\}^n} 2^{-n} g(v) p_\alpha(v)^2,
    \end{equation*}
    which is nonnegative whenever $g$ is nonnegative at every point of $\{0, 1\}^n$.
\end{proof}

Thus, the moment-matrix condition cannot make $y^*$ infeasible. The localizing matrices associated with inequalities valid for the 0--1 cube cannot make it infeasible either. In particular, the bound inequalities cannot cut off $y^*$. A cropping inequality $g_I(x) \geq 0$, however, is negative at the cube vertex that it cuts off. A certificate that the level-$t$ relaxation is empty is therefore a vector $\alpha$ indexed by the subsets of cardinality at most $t$ for which
\begin{equation*}
    \alpha^\top M_t(g_I \cdot y^*) \alpha = \sum_{v \in \{0, 1\}^n} g_I(v) 2^{-n} p_\alpha(v)^2 < 0.
\end{equation*}

We can reduce the search for such a certificate $\alpha$ to a single cropping inequality as we will show in the next lemma. For $I \subseteq N$, let $\gamma_I$ denote the coordinate map defined by
\begin{equation*}
    \gamma_I(x)_i :=
    \begin{cases}
      x_i, & i \in I,\\
      1 - x_i, & i \in N \setminus I.
    \end{cases}
\end{equation*}
     Applying $\gamma_I$ twice returns every coordinate to its original value. Thus, $\gamma_I$ is a self-inverse map on $\{0, 1\}^n$, and
\begin{equation*}
    g_I(\gamma_I(x)) = \sum_{i \in I} x_i + \sum_{i \in N \setminus I} x_i - \rho = \sum_{i=1}^n x_i - \rho = g_N(x).
\end{equation*}

Since $\gamma_I$ is a complementation, Lemma~\ref{lemma:gamma-matrix} gives an invertible matrix $C_{\gamma_I}$ such that $m_k(\gamma_I(x)) = C_{\gamma_I}\,m_k(x)$. Moreover, the inverse of $\gamma_I$ is itself, so applying \eqref{eq:gamma-matrix} twice gives
\begin{equation*}
    m_k(x) = m_k(\gamma_I(\gamma_I(x))) = C_{\gamma_I}^2\,m_k(x).
\end{equation*}
The entries of $m_k(x)$ form a basis for the space of multilinear polynomials of degree at most $k$. Therefore, the identity $m_k(x) = C_{\gamma_I}^2m_k(x)$ implies that $C_{\gamma_I}^2$ is the identity matrix. In particular, $C_{\gamma_I}^{-1} = C_{\gamma_I}$.

\begin{lemma} \label{lemma:certificate-symmetry}
    For every $I \subseteq N$ and every order $k$, there exists a vector $\alpha$ such that
    \begin{equation*}
        \alpha^\top M_k(g_I \cdot y^*)\alpha < 0
    \end{equation*}
    if and only if there exists a vector $\widetilde\alpha$ such that
    \begin{equation*}
        \widetilde\alpha^\top M_k(g_N \cdot y^*)\widetilde\alpha < 0,
    \end{equation*}
    where $g_N(x) = \sum_{i=1}^n x_i - \rho$.
\end{lemma}

\begin{proof}
    We first derive a relation between the corresponding quadratic forms. Fix $I \subseteq N$ and let $q$ be any vector indexed by $\{J \subseteq N: |J| \leq k\}$. Then,
    \begin{align*}
         q^\top M_k(g_N \cdot y^*)q
         & = 2^{-n} \sum_{v \in \{0, 1\}^n} g_N(v)\bigl(q^\top m_k(v)\bigr)^2 \\
         & = 2^{-n} \sum_{u \in \{0, 1\}^n} g_I(u)\bigl(q^\top m_k(\gamma_I(u))\bigr)^2 \\
         & = 2^{-n} \sum_{u \in \{0, 1\}^n} g_I(u)\bigl((C_{\gamma_I}^\top q)^\top m_k(u)\bigr)^2 \\
         & = (C_{\gamma_I}^\top q)^\top M_k(g_I \cdot y^*)(C_{\gamma_I}^\top q).
    \end{align*}    
    The second equality holds because $\gamma_I$ is a bijection of $\{0, 1\}^n$ and $g_N(\gamma_I(u)) = g_I(u)$.
    
    Now suppose that $\alpha$ is a certificate for $g_I$, such that $\alpha^\top M_k(g_I \cdot y^*)\alpha < 0$, and set $\widetilde\alpha := C_{\gamma_I}^\top \alpha$. Since $C_{\gamma_I}^2$ is the identity matrix, $(C_{\gamma_I}^\top)^2$ is also the identity matrix, and hence $C_{\gamma_I}^\top \widetilde\alpha = \alpha$. Applying the equality above with $q = \widetilde\alpha$ gives
    \begin{equation*}
        \widetilde\alpha^\top M_k(g_N \cdot y^*)\widetilde\alpha = \alpha^\top M_k(g_I \cdot y^*)\alpha < 0.
    \end{equation*}
    This proves that every certificate for $g_I$ produces a certificate for $g_N$. For the reverse direction, suppose that $\widetilde\alpha$ is a certificate for $g_N$ and set $\alpha := C_{\gamma_I}^\top \widetilde\alpha$. Applying the equality above with $q = \widetilde\alpha$ gives
    \begin{equation*}
        \alpha^\top M_k(g_I \cdot y^*)\alpha = \widetilde\alpha^\top M_k(g_N \cdot y^*)\widetilde\alpha < 0.
    \end{equation*}
\end{proof}

It therefore suffices to consider the single cropping inequality $g_N(x) = \sum_{i=1}^n x_i - \rho$.

\section{Reduction to a Least-Squares Problem} \label{sec:least-squares}

Having reduced the analysis to the single localizing matrix $M_t(g_N \cdot y^*)$, we now simplify the search for a negative certificate. We first show that it suffices to consider coefficient vectors that are constant on sets of the same cardinality. This reduces the quadratic form to an explicit $t$-dimensional optimization problem, which can then be written as a weighted least-squares problem.

\begin{lemma}\label{lemma:uniform-value-reduction}
    There exists a vector $\alpha$ indexed by $\{I \subseteq N: |I| \leq t\}$ such that $\alpha^\top M_t(g_N \cdot y^*)\alpha < 0$ if and only if there exists such a vector $\widetilde\alpha$ satisfying
    \begin{equation*}
        \widetilde\alpha_I = \widetilde\alpha_J \qquad \text{whenever } |I| = |J|.
    \end{equation*}
\end{lemma}

\begin{proof}
    The reverse implication is immediate, so it suffices to prove the forward implication.

    Given a certificate $\alpha$, define
    \begin{equation*}
        \widetilde\alpha_I := \frac{1}{\binom{n}{|I|}} \sum_{\substack{J \subseteq N \\ |J| = |I|}} \alpha_J, \qquad |I| \leq t.
    \end{equation*}
    By construction, $\widetilde\alpha_I$ depends only on $|I|$. Fix $r \in \{0, \ldots, n\}$ and let
    \begin{equation*}
        V_r := \left\{ v \in \{0, 1\}^n: \sum_{i=1}^n v_i = r \right\}.
    \end{equation*}
    For a set $I \subseteq N$ with $|I| = \ell$, the monomial $\prod_{i \in I} v_i$ equals one for exactly $\binom{n - \ell}{r - \ell}$ vectors $v \in V_r$. Therefore,
    \begin{align*}
        \sum_{v \in V_r} p_\alpha(v) &= \sum_{\ell = 0}^t \binom{n - \ell}{r - \ell} \sum_{\substack{I \subseteq N: |I| = \ell}} \alpha_I \\
        & = \sum_{\ell = 0}^t \binom{n - \ell}{r - \ell} \sum_{\substack{I \subseteq N: |I| = \ell}} \widetilde\alpha_I \\
        & = \sum_{v \in V_r} p_{\widetilde\alpha}(v),
    \end{align*}
    where $\binom{n - \ell}{r - \ell} = 0$ when $\ell > r$. Moreover, 
    \begin{equation*}
        \sum_{v \in V_r} p_{\widetilde\alpha}(v)^2 
        = \frac{1}{|V_r|}\left(\sum_{v \in V_r} p_{\widetilde\alpha}(v)\right)^2
        = \frac{1}{|V_r|}\left(\sum_{v \in V_r} p_\alpha(v)\right)^2 
        \leq \sum_{v \in V_r} p_\alpha(v)^2,
    \end{equation*}
    where the first equality holds because $p_{\widetilde\alpha}(v)$ has the same value for every $v \in V_r$, the second equality holds because $p_\alpha$ and $p_{\widetilde\alpha}$ have the same sum over $V_r$, and the inequality is Cauchy--Schwarz.

    For $r = 0$, equality holds because $\widetilde\alpha_\emptyset = \alpha_\emptyset$. Since $g_N(v) = r - \rho$ for every $v \in V_r$, and $r - \rho > 0$ for every $r \geq 1$, it follows that
    \begin{align*}
        \widetilde\alpha^\top M_t(g_N \cdot y^*)\widetilde\alpha &= 2^{-n} \sum_{r=0}^n \left(r - \rho\right) \sum_{v \in V_r} p_{\widetilde\alpha}(v)^2 \\
        & \leq 2^{-n} \sum_{r=0}^n \left(r - \rho\right) \sum_{v \in V_r} p_\alpha(v)^2 \\
        & = \alpha^\top M_t(g_N \cdot y^*)\alpha.
    \end{align*}
    Thus, if $\alpha^\top M_t(g_N \cdot y^*)\alpha < 0$, then also $\widetilde\alpha^\top M_t(g_N \cdot y^*)\widetilde\alpha < 0$.
\end{proof}

We may further assume that $\alpha_\emptyset = 1$. Indeed, if $\alpha_\emptyset = 0$, then $p_\alpha(\mathbf{0}) = 0$, where $\mathbf{0}$ denotes the all-zero vector. Since $g_N(v) = |\operatorname{supp}(v)| - \rho > 0$ for every $v \in \{0, 1\}^n \setminus \{\mathbf{0}\}$, we have

\begin{equation*}
    \alpha^\top M_t(g_N \cdot y^*)\alpha = 2^{-n} \sum_{v \in \{0, 1\}^n} g_N(v)p_\alpha(v)^2 = 2^{-n} \sum_{v \neq \mathbf{0}} g_N(v)p_\alpha(v)^2 \geq 0.
\end{equation*}

Thus every negative certificate has $\alpha_\emptyset \neq 0$. Replacing $\alpha$ by $\alpha / \alpha_\emptyset$ preserves the negativity of the quadratic form and normalizes the certificate so that $\alpha_\emptyset = 1$.

By Lemma~\ref{lemma:uniform-value-reduction}, we may assume that $\alpha_I$ depends only on $|I|$. Define $z \in \R^t$ by $z_j := \alpha_I$ for every $I \subseteq N$ with $|I| = j$, where $j = 1, \ldots, t$. If $v \in \{0, 1\}^n$ satisfies $\sum_{i=1}^n v_i = r$, then exactly $\binom{r}{j}$ monomials of degree $j$ evaluate to one at $v$. Consequently,
$p_\alpha(v) = 1 + \sum_{j=1}^t \binom{r}{j}z_j$. Note that this expression is now stated in terms of $z$ and $r$. It will be denoted by $p_z(r)$ in the next section.

There are $\binom{n}{r}$ vectors $v \in \{0, 1\}^n$ satisfying $\sum_i v_i = r$, and for each such vector, $g_N(v) = r - \rho$.

Therefore,
\begin{equation} \label{eq:reduced-objective}
        2^n \alpha^\top M_t(g_N \cdot y^*)\alpha = -\rho + \sum_{r=1}^n \binom{n}{r}(r - \rho) \left(1 + \sum_{j=1}^t \binom{r}{j} z_j \right)^2 =: f_{n,t}(z).
\end{equation}
For $r = 1, \ldots, n$, define $w_r := \binom{n}{r}(r - \rho)$. For each $t$, let $A_t \in \R^{n \times t}$ and $b \in \R^n$ be defined by
\begin{equation} \label{eq:Arj}
    (A_t)_{r,j} := \sqrt{w_r}\binom{r}{j}, \qquad b_r := -\sqrt{w_r},
\end{equation}
for $r = 1, \ldots, n$ and $j = 1, \ldots, t$. Then
\begin{equation*}
    (A_t z - b)_r = \sqrt{w_r} \left(1 + \sum_{j=1}^t \binom{r}{j} z_j \right),
\end{equation*}
so that $f_{n,t}(z) = \lVert A_t z - b \rVert_2^2 - \rho$. We can now state the least-squares characterization.

\begin{theorem}\label{theorem:least-squares}
    The Lasserre rank of $Q_{n,\rho}$ is the smallest integer $t$ such that
    \begin{equation*}
        \min_{z \in \R^t} \lVert A_t z - b \rVert_2^2 < \rho,
    \end{equation*}
    where $A_t$ and $b$ are defined in \eqref{eq:Arj}.
\end{theorem}

\begin{proof}
    By Lemmas~\ref{lemma:symmetry}, \ref{lemma:moment-psd}, and \ref{lemma:certificate-symmetry}, the level-$t$ Lasserre relaxation is empty if and only if there exists $\alpha$ such that $\alpha^\top M_t(g_N \cdot y^*)\alpha < 0$. By Lemma~\ref{lemma:uniform-value-reduction}, we may restrict attention to vectors whose entries depend only on the cardinality of the indexing set. Moreover, every negative certificate has $\alpha_\emptyset \neq 0$, so we may normalize to $\alpha_\emptyset = 1$. Thus such certificates are parametrized by $z \in \R^t$, and by~\eqref{eq:reduced-objective},
    \begin{equation*}
        2^n \alpha^\top M_t(g_N \cdot y^*)\alpha = \lVert A_tz - b \rVert_2^2 - \rho.
    \end{equation*}
    Hence the level-$t$ relaxation is empty if and only if
    \begin{equation*}
        \min_{z \in \R^t} \lVert A_tz - b \rVert_2^2 < \rho.
    \end{equation*}
    The result follows by taking the smallest such $t$.
\end{proof}

For $\rho=1/2$, solving the least-squares problems in Theorem~\ref{theorem:least-squares} gives Lasserre ranks $6$, $53$, and $131$ for $n = 10$, $100$, and $250$, respectively. For larger values of $n$, the weights $w_r = \binom{n}{r}(r-\rho)$ grow rapidly, making the direct least-squares formulation numerically intractable.

\section{A Recurrence for Computing the Lasserre Rank} \label{sec:rank-computation}

As shown in Theorem~\ref{theorem:least-squares}, computing the Lasserre rank of the cropped hypercube $Q_{n,\rho}$ amounts to finding the first integer $t$ for which the minimum of $f_{n,t}(z)$ is negative. Using \eqref{eq:reduced-objective}, we have 
\begin{equation} \label{eq:function}
    f_{n,t}(z) = \sum_{r=0}^n \binom{n}{r}(r - \rho) p_z(r)^2,
\end{equation}
where $p_z(0) = 1$ and, for $0 \leq t \leq n$ and $z \in \R^t$, $p_z$ is the univariate polynomial
\begin{equation} \label{eq:p-z-definition-rank}
    p_z(x) := 1 + \sum_{j=1}^t \binom{x}{j} z_j.
\end{equation}
The polynomial $p_z$ is the univariate version of the symmetric multivariate polynomial $p_\alpha$ introduced in Section~\ref{sec:symmetry}. For the associated coefficient vectors $\alpha$ and $z$, we have $p_\alpha(v) = p_z(r)$ whenever $\sum_{i=1}^n v_i = r$.
In the remainder of this paper, we will need to extend the formula $\binom{x}{j} := x(x - 1)\cdots(x - j + 1)/j!$ to nonintegral scalars $x$. When $t = 0$, the sum in \eqref{eq:p-z-definition-rank} is empty and $p_z(x) = 1$.

Denote the minimum of $f_{n,t}$ by $f_{n,t}^* := \min_{z \in \R^t} f_{n,t}(z)$. 
The next theorem shows that the first negative value of $f_{n,t}^*$ can be identified through a scalar recurrence.

\begin{theorem} \label{theorem:exact-lasserre-rank}
    For every $n \geq 2$ and $0<\rho<1$, define
    \begin{equation} \label{eq:recurrence}
        \Delta_{-1} := 1, \qquad \Delta_0 := n - 2\rho, \qquad \Delta_t := (n - 2\rho)\Delta_{t - 1} - t(n - t + 1)\Delta_{t - 2} \quad (1 \leq t \leq n).
    \end{equation}
    Then
    \begin{equation}\label{eq:exact-rank-recurrence}
        \operatorname{rank}_{\LAS}(Q_{n,\rho}) = \min \bigl\{0 \leq t \leq n \mathrel{\big|} \Delta_t < 0\bigr\}.
    \end{equation}
    In particular, the exact rank can be computed using at most $n$ recurrence steps. 
\end{theorem}

As an example, consider $n = 3$ and $\rho=1/2$. The recurrence in \eqref{eq:recurrence} gives
\begin{equation*}
    \Delta_0 = 2, \qquad \Delta_1 = 1, \qquad \Delta_2 = -6.
\end{equation*}
It follows from Theorem~\ref{theorem:exact-lasserre-rank} that the Lasserre rank of $Q_{3,1/2}$ is $2$. This matches the rank observed in Figure~\ref{fig:Q_3-lasserre}. 

The endpoint $t = n$ in~\eqref{eq:exact-rank-recurrence} is essential. For example, when $n=2$ and $\rho=1/4$, the recurrence gives $\Delta_0=3/2$, $\Delta_1=1/4$, and $\Delta_2=-21/8$, so the rank is $2$. 

\bigskip
To derive Theorem~\ref{theorem:exact-lasserre-rank}, we relate the minimization problem $\min_{z \in \R^t} f_{n,t}(z)$ to the sign change of a determinant. We first introduce the underlying square matrix $H_t$, as well as a related family of polynomials (\cite{krawtchouk}).

\bigskip
Fix integers $n \geq 2$ and $0 \leq t \leq n$. Let $H_t$ be the $(t + 1) \times (t + 1)$ symmetric tridiagonal matrix indexed by $0, \ldots, t$ whose nonzero entries are
\begin{align}
    (H_t)_{jj} & = n - 2\rho && (0 \leq j \leq t), \label{eq:definition-Ht-diagonal} \\
    (H_t)_{j,j + 1} = (H_t)_{j + 1,j} & = -\sqrt{(j + 1)(n - j)} && (0 \leq j \leq t - 1). \label{eq:definition-Ht-off-diagonal}
\end{align}

Let $0 \leq j \leq n$  be an integer. The {\it binary Krawtchouk polynomials of degree $j$} are defined as 
\begin{equation} \label{eq:krawtchouk-definition}
    K_j^{n}(x) := \sum_{i=0}^{j} (-1)^i \binom{x}{i} \binom{n-x}{j-i}.
\end{equation}
Throughout this paper, we refer to these polynomials simply as Krawtchouk polynomials. They can be obtained using a generating function:
\begin{equation} \label{eq:krawtchouk-definition-generating}
    K_j^{n}(x) = [u^j](1 - u)^x(1 + u)^{n - x}, \qquad 0 \leq j \leq n,
\end{equation}
where $[u^j]F(u)$ denotes the coefficient of $u^j$ in the power series of $F$ (see \cite{macwilliams1977theory} Chapter~5, Section~7, p.~151). We work with the normalized polynomials for fixed $n$
\begin{equation*}
    \phi_j(x) := \frac{K_j^{n}(x)}{\sqrt{\binom{n}{j}}}, \qquad 0 \leq j \leq n.
\end{equation*}

The following properties from Chapter~5, Section~7, pp.~150--152 of \cite{macwilliams1977theory} will be used throughout this section.
\begin{enumerate}[label=\roman*.]
    \item The polynomial $\phi_j$ has degree $j$ and satisfies $\phi_j(0) = \sqrt{\binom{n}{j}}$. 
    
    Consequently, $\phi_0, \ldots, \phi_t$ form a basis of the vector space $\R_t[x]$ of real polynomials of degree at most $t$.  
    \item The polynomials $\phi_0, \ldots, \phi_n$ are orthonormal with respect to the binomial distribution.
    \begin{equation}\label{eq:krawtchouk-orthogonality}
        \E_{X \sim \operatorname{Bin}(n, 1/2)}[\phi_i(X)\phi_j(X)] = 2^{-n} \sum_{r=0}^n \binom{n}{r} \phi_i(r)\phi_j(r) = \delta_{ij}, \qquad 0 \leq i,j \leq n,
    \end{equation}
    where $\delta_{ij} =1$ if $i=j$ and 0 otherwise.
    \item The polynomials satisfy the three-term recurrence
    \begin{equation}\label{eq:krawtchouk-recurrence}
        (n - 2x)\phi_j(x) = \sqrt{(j + 1)(n - j)}\,\phi_{j + 1}(x) + \sqrt{j(n - j + 1)}\,\phi_{j - 1}(x),
    \end{equation}
    for $0 \leq j \leq n$, where $\phi_{-1} := 0$. Note that when $j = n$, we set the term $\sqrt{(j + 1)(n - j)}\,\phi_{j + 1}(x) = 0$.
\end{enumerate}

Define $\nu \in \R^{t + 1}$ by $\nu_j := \sqrt{\binom{n}{j}}$ for $0 \leq j \leq t$.

\begin{lemma}\label{lem:reduce-to-constrained-minimization}
    For any integers $n \geq 2$ and $1 \leq t \leq n$,
    \begin{equation}\label{eq:constrained-H-rank}
        f_{n,t}^* = 2^{n - 1} \min \left\{\beta^{\top} H_t \beta \mathrel{\big|} \nu^{\top}\beta = 1,\ \beta \in \R^{t + 1}\right\}.
    \end{equation}
\end{lemma}
\begin{proofwithclaims}

    \begin{claim}
      There is a bijection between $\R^t$ and the affine space $\{\beta \in \R^{t + 1} \mid \nu^{\top}\beta = 1\}$. 
    \end{claim}    
    \begin{proofclaim}
        Given $z \in \R^t$, the basis property of $\phi_0, \ldots, \phi_t$ gives a unique vector $\beta = (\beta_0, \ldots, \beta_t)^{\top} \in \R^{t + 1}$ such that
        \begin{equation}\label{eq:p_z-expression}
            p_z(x) = \sum_{j=0}^t \beta_j\phi_j(x).
        \end{equation}
        Since $p_z(0) = 1$ and $\phi_j(0) = \nu_j$, we have $\nu^{\top}\beta = 1$.
        
        Conversely, let $\beta \in \R^{t + 1}$ satisfy $\nu^{\top}\beta = 1$ and define $q_\beta(x) := \sum_{j=0}^t \beta_j\phi_j(x)$. Then $q_\beta \in \R_t[x]$ and $q_\beta(0) = 1$. The polynomials $1, \binom{x}{1}, \ldots, \binom{x}{t}$ also form a basis of $\R_t[x]$, so there is a unique $z \in \R^t$ such that
        \begin{equation*}
            q_\beta(x) = 1 + \sum_{j=1}^t \binom{x}{j} z_j.
        \end{equation*}
        Thus, $z \mapsto \beta$ gives the claimed bijection.
    \end{proofclaim}


    For every $\beta \in \R^{t + 1}$, define $q_\beta(x) := \sum_{j=0}^t \beta_j\phi_j(x)$.
    \begin{claim}
        For any $0\leq t\leq n$ and any $\beta \in \R^{t+1}$,
        \begin{equation}\label{eq:H-quadratic-form-rank}
            \beta^{\top}H_t\beta = 2^{-n}\sum_{r=0}^n \binom{n}{r}(2r-2\rho)q_\beta(r)^2.
        \end{equation}
    \end{claim}    
   
    \begin{proofclaim}
        Let $X \sim \operatorname{Bin}(n,1/2)$. By the definition of the binomial distribution,
        \begin{equation*}
            2^{-n}\sum_{r=0}^n \binom{n}{r}(2r-2\rho)q_\beta(r)^2 = \E\bigl[(2X-2\rho)q_\beta(X)^2\bigr].
        \end{equation*}
        Decomposing the factor $2X-2\rho$ gives
        \begin{equation*}
            \E\bigl[(2X-2\rho)q_\beta(X)^2\bigr] = (n-2\rho)\E\bigl[q_\beta(X)^2\bigr] - \E\bigl[(n-2X)q_\beta(X)^2\bigr].
        \end{equation*}
        Equation~\eqref{eq:krawtchouk-orthogonality} and the definition of $q_\beta$ yield
        \begin{equation*}
            \E\bigl[q_\beta(X)^2\bigr] = \sum_{i,j=0}^t \beta_i\beta_j\E[\phi_i(X)\phi_j(X)] = \sum_{j=0}^t\beta_j^2.
        \end{equation*}
        Multiplying \eqref{eq:krawtchouk-recurrence} by $\phi_i(X)$ and taking expectations gives
        \begin{equation*}
            \E\bigl[\phi_i(X)(n-2X)\phi_j(X)\bigr] = \sqrt{(j+1)(n-j)}\,\delta_{i,j+1} + \sqrt{j(n-j+1)}\,\delta_{i,j-1}.
        \end{equation*}
        When $t = n$, for $j = n$ we use the terminal identity $(n-2X)\phi_n(X)=\sqrt{n}\,\phi_{n-1}(X)$.
        It follows that
        \begin{align*}
            \E\bigl[(n-2X)q_\beta(X)^2\bigr]
            & = \sum_{i,j=0}^t \beta_i\beta_j \E\bigl[\phi_i(X)(n-2X)\phi_j(X)\bigr] \\
            & = 2\sum_{j=0}^{t-1} \sqrt{(j+1)(n-j)}\,\beta_j\beta_{j+1},
        \end{align*}
        where the second equality follows by reindexing one of the two identical sums. Therefore,
        \begin{align*}
            \E\bigl[(2X-2\rho)q_\beta(X)^2\bigr]
            & = (n-2\rho)\sum_{j=0}^t\beta_j^2 - 2\sum_{j=0}^{t-1} \sqrt{(j+1)(n-j)}\,\beta_j\beta_{j+1} \\
            & = \beta^\top H_t\beta.
        \end{align*}
        Combining this with the first equality proves \eqref{eq:H-quadratic-form-rank}.
    \end{proofclaim}

    Finally, fix $z\in\R^t$ and let $\beta\in\R^{t+1}$ be the unique coefficient vector corresponding to $z$ under the bijection in Claim~1. Since $p_z=q_\beta$, equations~\eqref{eq:function} and \eqref{eq:H-quadratic-form-rank} give
    \begin{equation*}
        f_{n,t}(z) = \sum_{r=0}^n \binom{n}{r}(r-\rho)p_z(r)^2 = 2^{n-1}\beta^\top H_t\beta.
    \end{equation*}
    Together with the bijection from Claim~1, this proves \eqref{eq:constrained-H-rank}.
\end{proofwithclaims}

For fixed $n$, the sequence $(f_{n,t}^*)_{t = 0}^{n}$ is nonincreasing in $t$. Indeed, for $t \leq n-1$, appending a zero to any $z \in \R^t$ leaves both $p_z$ and the objective value unchanged at level $t + 1$.

The next lemma relates the sign of $f_{n,t}^*$ to that of $\det(H_t)$. We use
\begin{equation*}
    \operatorname{sgn}(a) :=
    \begin{cases}
        -1, & a < 0, \\
         0, & a = 0, \\
         1, & a > 0.
    \end{cases}
\end{equation*}

\begin{lemma} \label{lem:reduce-to-det-sign}
    For $n \geq 2$ and $1 \leq t \leq n$, we have
    \begin{equation*}
        \operatorname{sgn}(f_{n,t}^*) = \operatorname{sgn}(\det(H_t)).
    \end{equation*}
\end{lemma}

\begin{proof}
    We first show that $H_t$ is positive definite on $\nu^\perp := \{d \in \R^{t + 1} \mid \nu^\top d = 0\}$. Fix $0 \neq d \in \nu^\perp$. Recall the definition of $q_d(x)$:
    \begin{equation*}
        q_d(x) := \sum_{j=0}^t d_j\phi_j(x).
    \end{equation*}
    Since $\phi_j(0) = \sqrt{\binom{n}{j}} = \nu_j$, we have $q_d(0) = \nu^\top d = 0$. Therefore,
    \begin{equation} \label{eq:H-positive-on-nu-perp}
        \begin{aligned}
            d^{\top} H_t d
            = 2^{-n} \sum_{r=0}^n \binom{n}{r}(2r - 2\rho)q_d(r)^2
            = 2^{-n} \sum_{r=1}^n \binom{n}{r}(2r - 2\rho)q_d(r)^2,
        \end{aligned}
    \end{equation}
    where the first equality is \eqref{eq:H-quadratic-form-rank} and the second equality follows from $q_d(0) = 0$.
    Every coefficient in this sum is positive. If the sum were zero, then $q_d$ would vanish at $0, 1, \ldots, n$. Since $\deg(q_d) \leq t \leq n$, this would imply $q_d \equiv 0$. The linear independence of $\phi_0, \ldots, \phi_t$ would then imply $d = 0$, which is a contradiction. Hence $d^{\top} H_t d > 0$ for every $0 \neq d \in \nu^\perp$.

    Let $U \in \R^{(t + 1) \times t}$ have columns that form a basis of $\nu^\perp$. Then $G := U^{\top} H_t U$ is positive definite by \eqref{eq:H-positive-on-nu-perp}. Take, for example, the feasible point $\bar\beta := \nu/(\nu^{\top}\nu)$ in $\nu^{\top}\beta =1$. Every feasible point in $\nu^{\top}\beta = 1$ can be written uniquely as $\bar\beta + U\xi$ for some $\xi \in \R^t$. Under this parametrization, the objective in \eqref{eq:constrained-H-rank} becomes
    \begin{equation*}
        (\bar\beta + U\xi)^{\top} H_t (\bar\beta + U\xi) = \xi^{\top} G \xi + 2h^{\top} \xi + \bar\beta^{\top} H_t \bar\beta,
    \end{equation*}
    where $h := U^{\top} H_t \bar\beta$. Since $G$ is positive definite, this quadratic function has the unique minimizer $\xi^* = -G^{-1}h$. Therefore, the problem in \eqref{eq:constrained-H-rank} has the unique minimizer $\beta^* := \bar\beta + U\xi^*$.

    The first-order optimality conditions for the Lagrangian $\L(\beta, \lambda) = \beta^{\top} H_t \beta - 2\lambda(\nu^{\top}\beta - 1)$ give
    \begin{equation*}
        H_t\beta^* = \lambda \nu
    \end{equation*}
    for some $\lambda \in \R$. Since $\nu^{\top}\beta^* = 1$, we obtain $(\beta^*)^{\top} H_t \beta^* = \lambda$. Equation~\eqref{eq:constrained-H-rank} shows that $\operatorname{sgn}(f_{n,t}^*) = \operatorname{sgn}(\lambda)$

    Define $P := [\,\beta^*\;U\,] \in \R^{(t + 1) \times (t + 1)}$. The matrix $P$ is invertible because $\beta^* \notin \nu^\perp$ and the columns of $U$ form a basis of $\nu^\perp$. Moreover,
    \begin{equation*}
        P^{\top} H_t P =
        \begin{pmatrix}
            (\beta^*)^{\top} H_t \beta^* & (\beta^*)^{\top} H_t U \\
            U^{\top} H_t \beta^* & U^{\top} H_t U
        \end{pmatrix}
        =
        \begin{pmatrix}
            \lambda & 0 \\
            0 & G
        \end{pmatrix},
    \end{equation*}
    where $U^{\top} H_t \beta^* = \lambda U^{\top} \nu = 0$. Taking determinants yields
    \begin{equation*}
        (\det P)^2 \det(H_t) = \lambda \det(G).
    \end{equation*}
    Both $(\det P)^2$ and $\det(G)$ are positive, which proves the result.
\end{proof}

We next derive a recurrence for $\det(H_t)$ and relate it to a Krawtchouk polynomial.

\begin{lemma} \label{lem:determinant-recurrence}
    Let $\Delta_t := \det(H_t)$ for $0\leq t\leq n$ and set $\Delta_{-1} := 1$. Then $\Delta_0 = n - 2\rho$ and, for $1 \leq t \leq n$,
    \begin{equation}\label{eq:Delta-recurrence}
        \Delta_t = (n - 2\rho)\Delta_{t - 1} - t(n - t + 1)\Delta_{t - 2}.
    \end{equation}
    Moreover, for $0\leq t\leq n-1$,
    \begin{equation}\label{eq:connection-det-kraw}
        \det(H_t) = (t + 1)!K_{t + 1}^{n}(\rho).
    \end{equation}
\end{lemma}

\begin{proof}
    Expanding the determinant of the tridiagonal matrix $H_t$ along its last row gives the standard determinant recurrence (\cite{usmani1994inversion}, pp.~413--414)
    \begin{equation*}
        \Delta_t = (H_t)_{tt}\Delta_{t - 1} - (H_t)_{t,t - 1}^2\Delta_{t - 2}.
    \end{equation*}
    Equations~\eqref{eq:definition-Ht-diagonal} and \eqref{eq:definition-Ht-off-diagonal} give $(H_t)_{tt} = n - 2\rho$ and $(H_t)_{t,t - 1}^2 = t(n - t + 1)$, which proves \eqref{eq:Delta-recurrence}.

    The unnormalized Krawtchouk recurrence given in Chapter~5, Section~7, p.~151 of \cite{macwilliams1977theory} is
    \begin{equation*}
        (n - 2x)K_j^{n}(x) = (j + 1)K_{j + 1}^{n}(x) + (n - j + 1)K_{j - 1}^{n}(x).
    \end{equation*}
    For $0\leq t\leq n-1$, setting $x = \rho$ and $j = t$ yields
    \begin{equation*}
        (t + 1)K_{t + 1}^{n}(\rho) = (n - 2\rho)K_t^{n}(\rho) - (n - t + 1)K_{t - 1}^{n}(\rho).
    \end{equation*}
    For $1\leq t\leq n-1$, multiplication by $t!$ gives
    \begin{equation*}
        (t+1)!K_{t+1}^{n}(\rho) = (n-2\rho)t!K_t^{n}(\rho) - t(n-t+1)(t-1)!K_{t-1}^{n}(\rho).
    \end{equation*}
    The corresponding initial values are $K_0^{n}(\rho) = 1 = \Delta_{-1}$ and $K_1^{n}(\rho) = n-2\rho = \Delta_0$. Thus, the quantities $(t+1)!K_{t+1}^{n}(\rho)$ satisfy the same recurrence and initial conditions as $\Delta_t$, so $\Delta_t=(t+1)!K_{t+1}^{n}(\rho)$ for $0 \leq t \leq n-1$. This proves \eqref{eq:connection-det-kraw}.
\end{proof}

\bigskip
We are now ready to prove Theorem~\ref{theorem:exact-lasserre-rank}.

\begin{proof}[Proof of Theorem~\ref{theorem:exact-lasserre-rank}]

    By Theorem~\ref{theorem:least-squares} and Lemmas~\ref{lem:reduce-to-det-sign} and \ref{lem:determinant-recurrence}, for $1 \leq t \leq n$,
    \begin{equation*}
        \operatorname{sgn}(f_{n,t}^*) = \operatorname{sgn}(\det(H_t)) = \operatorname{sgn}(\Delta_t).
    \end{equation*}
    The same conclusion holds for $t = 0$ because $f_{n,0}^* = 2^{n - 1}(n - 2\rho) > 0$ and $\Delta_0 = n - 2\rho > 0$.

    Finally, at $t=n$, choosing $z_j=(-1)^j$ gives
    \begin{equation*}
        p_z(r) = \sum_{j=0}^r(-1)^j\binom{r}{j} = 0, \qquad r = 1,\ldots,n.
    \end{equation*}
    Hence $f_{n,n}^* = -\rho<0$, so the set on the right-hand side of~\eqref{eq:exact-rank-recurrence} is nonempty. Theorem~\ref{theorem:least-squares} now shows that the rank is precisely the first $t$ for which $\Delta_t < 0$. This proves \eqref{eq:exact-rank-recurrence}.
\end{proof}

To bound the bit complexity for fixed rational $\rho = p/q$, with positive integers $p<q$, set $\widehat\Delta_t := q^{t+1}\Delta_t$. The initial values are $\widehat\Delta_{-1} = 1$ and $\widehat\Delta_0 = qn - 2p$, and
\begin{equation*}
    \widehat\Delta_t = (qn-2p)\widehat\Delta_{t-1} - q^2t(n-t+1)\widehat\Delta_{t-2}.
\end{equation*}
This scaling preserves signs and gives an integer recurrence. Let $B_t := \max_{-1\leq j\leq t} |\widehat\Delta_j|$. Since $p$ and $q$ are fixed and $t(n-t+1) \leq n^2$, the recurrence gives $B_t \leq C_{p,q}n^2B_{t-1}$ for a constant $C_{p,q}$, and hence each $\widehat\Delta_t$ has $O(n\log n)$ bits. Using standard integer arithmetic, each recurrence step takes $O(n\log^2 n)$ time, so all $\widehat\Delta_t$ can be computed in $O(n^2\log^2 n)$ time.


\bigskip

If we allow $\rho$ to vary with $n$, the Lasserre rank can be equal to $n$. Let $\rho_n$ denote the largest value of $\rho$ for which $Q_{n,\rho}$ has Lasserre rank $n$. This small-$\rho$ regime was previously studied in \cite{kurpisz2019boolean,au2018elementary}. Using $1/\rho$ as the parameter, \cite[Theorem~12 and Lemma~13]{kurpisz2019boolean} obtained lower and upper bounds on the rank when $1/\rho \geq 2$. In particular, \cite[Corollary~4.2]{kurpisz2019boolean} showed that $Q_{n,2^{-(n+1)}}$ has Lasserre rank $n$. More directly, (\cite{au2018elementary}, Theorem~14) proved that
\begin{equation*}
    \frac{n+1}{2^{n+2}-n-3} \leq \rho_n \leq \frac{n}{2^{n+1}-2}.
\end{equation*}
Thus, $\rho_n = \Theta(n/2^n)$, but these bounds differ asymptotically by a factor of two. The following theorem characterizes $\rho_n$ exactly and shows that the upper bound in \cite{au2018elementary} is asymptotically tight.

\begin{theorem} \label{prop:full-rank-threshold}
    For every $n \geq 2$, there is a unique $\rho_n \in (0,1)$ satisfying
    \begin{equation} \label{eq:full-rank-threshold}
        \rho_n \sum_{r=1}^n\frac{\binom{n}{r}}{r-\rho_n} = 1.
    \end{equation}
    Moreover, $\operatorname{rank}_{\LAS}(Q_{n,\rho})=n$ if and only if $0 < \rho \leq \rho_n$, and $\rho_n \sim n/2^{n+1}$.
\end{theorem}

\begin{proof}
    For $0\leq\rho<1$, let $T_n(\rho) := \sum_{r=1}^n \binom{n}{r} / (r-\rho)$. By Theorem~\ref{theorem:least-squares} and the monotonicity of $f_{n,t}^*$, the rank is $n$ if and only if $f_{n,n-1}^* \geq 0$. Consider the matrix $A_{n-1} \in \R^{n\times(n-1)}$ and vector $b \in \R^n$ defined in~\eqref{eq:Arj}, so that $f_{n,n-1}(z) = \lVert A_{n-1}z-b \rVert_2^2 - \rho$. The first $n-1$ rows of $A_{n-1}$ form an invertible lower triangular matrix, and hence $A_{n-1}$ has rank $n-1$. Define $\eta \in \R^n$ by $\eta_r := (-1)^r\sqrt{\binom{n}{r}/(r-\rho)}$.
    For every $j=1,\ldots,n-1$, the binomial theorem gives
    \begin{equation*}
        (A_{n-1}^\top\eta)_j
        = \sum_{r=j}^n(-1)^r\binom{n}{r}\binom{r}{j}
        = (-1)^j\binom{n}{j}
          \sum_{s=0}^{n-j}(-1)^s\binom{n-j}{s}
        = 0.
    \end{equation*}
    Thus $\ker(A_{n-1}^\top) = \operatorname{span}\{\eta\}$. Moreover, $\lVert \eta \rVert_2^2 = T_n(\rho)$ and $\eta^\top b = -\sum_{r=1}^n(-1)^r\binom{n}{r} = 1$. If $z^*$ minimizes $f_{n,n-1}$, the normal equations imply that $A_{n-1}z^*-b = \lambda \eta$ for some $\lambda \in \R$. Taking the inner product with $\eta$ yields $\lambda T_n(\rho) = -1$, and therefore
    \begin{equation*}
        f_{n,n-1}^* = \lambda^2 \lVert \eta \rVert_2^2 - \rho = \frac{1}{T_n(\rho)} - \rho.
    \end{equation*}
    Thus the Lasserre rank is $n$ if and only if $\rho T_n(\rho) \leq 1$. Since $\rho T_n(\rho)$ is continuous and strictly increasing from zero to infinity on $(0,1)$, this proves the existence and uniqueness of $\rho_n$ and the claimed characterization of $\operatorname{rank}_{\LAS}(Q_{n,\rho}) = n$.

    It remains to estimate $\rho_n$. Since $r(1-\rho_n) \leq r-\rho_n \leq r$ for every $r \geq 1$, we have $T_n(0) \leq T_n(\rho_n) \leq T_n(0)/(1-\rho_n)$. Together with $\rho_nT_n(\rho_n)=1$, this gives
    \begin{equation*}
        1-\rho_n \leq \rho_n T_n(0) \leq 1.
    \end{equation*}
    Also, $T_n(0) \geq n$ implies $\rho_n \leq T_n(0)^{-1} \leq 1/n$, so $\rho_n \to 0$. Hence $\rho_n T_n(0) \to 1$, or equivalently, $\rho_n \sim T_n(0)^{-1}$. To estimate $T_n(0)^{-1}$, let $X_n \sim \operatorname{Bin}(n,1/2)$. Then
    \begin{equation*}
        T_n(0) = 2^n\E\!\left[ X_n^{-1}\mathbf{1}_{\{X_n\geq1\}} \right],
    \end{equation*}
    where the expression inside the expectation is defined to be zero when $X_n=0$. Set $E_n:=\{|X_n-n/2|\leq n^{3/4}\}$ and let $E_n^c$ denote the complement set. 
    A Chernoff bound gives $\Pr(E_n^c) \leq 2e^{-2\sqrt n}$. On $E_n$, we have $X_n^{-1} = (2/n)(1+o(1))$ uniformly, and hence $\E[X_n^{-1}\mathbf{1}_{E_n}] = (2/n)(1+o(1))$. On the other hand, $X_n^{-1} \leq 1$ whenever $X_n \geq 1$, so
    \begin{equation*}
        0 \leq \E\!\left[ X_n^{-1}\mathbf{1}_{\{X_n\geq1\}\cap E_n^c} \right] \leq\Pr(E_n^c) \leq 2e^{-2\sqrt n} = o(n^{-1}).
    \end{equation*}
    Consequently, $T_n(0) \sim 2^{n+1}/n$, and thus $\rho_n \sim T_n(0)^{-1} \sim n/2^{n+1}$.
\end{proof}


\section{Asymptotically Tight Lasserre Rank} \label{sec:asymptotic-analysis}

We next turn to the asymptotic analysis for fixed $\rho$. The cropping constant is fixed independently of $n$ in the following theorem.

\begin{theorem} \label{thm:asymptotic}
    For every fixed $0<\rho<1$,
    \begin{equation}\label{eq:rank-asymptotic}
        \operatorname{rank}_{\LAS}(Q_{n,\rho}) = \frac{n}{2} + c_\rho\sqrt{n} + o(\sqrt{n}),
    \end{equation}
    where $c_\rho > 0$ is the unique zero of the function
    \begin{equation}\label{eq:kappa-definition-rank}
        \kappa_\rho(c):=\int_0^\infty s^\rho e^{-s^2/8} \cos \left(cs + \frac{\pi\rho}{2}\right) \, ds.
    \end{equation}
    In particular, $c_{1/2} \approx 0.3825$.
\end{theorem}

Lemmas~\ref{lem:reduce-to-det-sign} and \ref{lem:determinant-recurrence} reduce the rank computation to the sign of $K_{t + 1}^{n}(\rho)$ for $0\leq t\leq n-1$. The asymptotic behavior of Krawtchouk polynomials has been studied in several parameter ranges (\cite{ismail1998asymptotics,li2000uniform,qiu2004expansion,dai2007global,dominici2008krawtchouk}). 
The proof of the following lemma is fairly standard but we could not find a good reference for it, so we include it for completeness.

\begin{lemma} \label{lem:integration-form}
    For $n \geq 2$ and $0 \leq m \leq n$,
    \begin{equation}\label{eq:kraw-integral-rank}
        K_m^{n}(\rho) = \frac{2^{n + 1}}{\pi} \int_0^{\pi/2} (\cos \theta)^{n - \rho}(\sin \theta)^\rho\,\cos\!\left((n - 2m)\theta - \frac{\pi\rho}{2}\right) \, d\theta.
    \end{equation}
\end{lemma}

\begin{proof}
    The integral representation of the Krawtchouk polynomials \cite[p.~232, Eq.~(2.7) with  $p=1/2$ and $z=-2u$]{johansson2002nonintersecting}  is
    \begin{equation*}
        K_m^{n}(x) = \frac{1}{2\pi i}\oint_{|u|=R} (1+u)^{n-x}(1-u)^x\,\frac{du}{u^{m+1}}, \qquad 0<R<1.
    \end{equation*}
    Set $x=\rho$ and parametrize by $u=Re^{i\varphi}$. The factors in the integrand extend continuously to the boundary of the unit disk, and the parametrized integrand is uniformly bounded for $R\geq1/2$. Thus, by dominated convergence, letting $R\uparrow1$ gives
    \begin{equation*}
        K_m^{n}(\rho) = \frac{1}{2\pi}\int_{-\pi}^{\pi} (1+e^{i\varphi})^{n-\rho}(1-e^{i\varphi})^\rho\, e^{-im\varphi}\,d\varphi.
    \end{equation*}
    For $0 < \varphi < \pi$, we have $1+e^{i\varphi} = 2\cos(\varphi/2)e^{i\varphi/2}$ and $1-e^{i\varphi} = 2\sin(\varphi/2)e^{i(\varphi/2-\pi/2)}$. Pairing the conjugate integrands at $\varphi$ and $-\varphi$ yields twice the real part, hence
    \begin{equation*}
        K_m^{n}(\rho) = \frac{2^n}{\pi}\int_0^\pi \cos^{n-\rho}(\varphi/2)\sin^\rho(\varphi/2)\, \cos\left(\frac{n-2m}{2}\varphi-\frac{\pi\rho}{2}\right) \,d\varphi.
    \end{equation*}
    Substituting $\varphi=2\theta$ proves \eqref{eq:kraw-integral-rank}.
\end{proof}

Next we focus on the case when $m = \frac{n}{2} + c_n \sqrt{n}$ where $c_n \in \R$.

\begin{lemma} \label{lem:critical-window-rank}
    Let $m = m(n) \in \{0, \ldots, n\}$ for every integer $n \geq 2$, and write
    \begin{equation*}
        m = \frac{n}{2} + c_n\sqrt{n}, \qquad c_n \in \R.
\end{equation*}
    If $c_n \to c \in \R$, then
    \begin{equation}\label{eq:critical-window-limit-rank}
       \pi 2^{-n +\rho}n^{(\rho+1)/2}K_m^{n}(\rho) \longrightarrow \kappa_\rho(c),
    \end{equation}
    where $\kappa_\rho$ is the function \eqref{eq:kappa-definition-rank}. Moreover, $\kappa_\rho$ has a unique zero $c_\rho$, and
    \begin{equation*}
        \operatorname{sgn}(\kappa_\rho(c)) = \operatorname{sgn}(c_\rho - c).
    \end{equation*}
\end{lemma}

\begin{proof}
    Since $n-2m=-2c_n\sqrt{n}$, substituting $\theta=s/(2\sqrt{n})$ in~\eqref{eq:kraw-integral-rank} gives
    \begin{equation*}
        \pi 2^{-n+\rho}n^{(\rho+1)/2}K_m^{n}(\rho) = \int_0^{\pi\sqrt{n}} s^\rho \left(\cos\frac{s}{2\sqrt{n}}\right)^{n-\rho} \left( \frac{\sin(s/(2\sqrt{n}))}{s/(2\sqrt{n})} \right)^\rho \cos\left(c_ns+\frac{\pi\rho}{2}\right),ds.
    \end{equation*}
    We interpret the sine quotient as one at $s=0$ and extend the integrand by zero outside $[0,\pi\sqrt{n}]$. For every fixed $s\geq0$, the integrand converges to
    \begin{equation*}
        s^\rho e^{-s^2/8} \cos\left(cs+\frac{\pi\rho}{2}\right).
    \end{equation*}
    On $0 \leq x < \tfrac{\pi}{2}$, the inequalities $0 \leq \frac{\sin x}{x} \leq 1$ and $\log(\cos x) \leq -x^2/2$ imply
    \begin{equation*}
        \left(\cos\frac{s}{2\sqrt{n}}\right)^{n-\rho} \leq \exp\left(-\frac{n-\rho}{8n}s^2\right) \leq e^{-s^2/16}
    \end{equation*}
    for $n\geq2$. Thus, the absolute value of the extended integrand is bounded by the integrable function $s^\rho e^{-s^2/16}$. Therefore, \eqref{eq:critical-window-limit-rank} follows from the dominated convergence theorem.
    
    The cosine integral representation of the parabolic-cylinder function $D_\nu$ in \cite[Section~8.3, Eq.~(4), p.~120]{erdelyi1953higher}, applied with $\nu = \rho$, $z = -2c$, and the change of variables $s=2r$, gives
    \begin{equation*}
        \kappa_\rho(c) = 2^\rho\sqrt{2\pi}\,e^{-c^2}D_\rho(-2c).
    \end{equation*}
    For $0<\rho<1$, the function $D_\rho$ has exactly one real zero \cite[Section~8.6, p.~126]{erdelyi1953higher}; denote it by $x_\rho^*$. The same integral representation at $z=0$ shows that $D_\rho(0)>0$. 
    By \cite[Section~8.4, Eq.~(2), p.~123]{erdelyi1953higher},
    $D_\rho(-y)<0$ for all sufficiently large $y$. Therefore by continuity we get $x_\rho^*<0$ and
    \begin{equation*}
        D_\rho(x) < 0 \quad\text{for }x < x_\rho^*,
        \qquad
        D_\rho(x) > 0 \quad\text{for }x > x_\rho^*.
    \end{equation*}
    Set $c_\rho:=-x_\rho^*/2>0$. The factor $2^\rho\sqrt{2\pi}\,e^{-c^2}$ is positive, which gives the desired sign relation $\operatorname{sgn}(\kappa_\rho(c))=\operatorname{sgn}(c_\rho-c)$.
\end{proof}

We now prove Theorem~\ref{thm:asymptotic}.

\begin{proof}[Proof of Theorem~\ref{thm:asymptotic}]
    It remains to prove \eqref{eq:rank-asymptotic}. Write $r_n := \operatorname{rank}_{\LAS}(Q_{n,\rho})$. Fix $\varepsilon > 0$ and, for all sufficiently large $n$, set
    \begin{equation*}
        m_n^- := \left\lfloor \frac{n}{2} + (c_\rho - \varepsilon)\sqrt{n}\right\rfloor, \qquad m_n^+ := \left\lceil \frac{n}{2} + (c_\rho + \varepsilon)\sqrt{n}\right\rceil.
    \end{equation*}
    These integers lie in $\{1, \ldots, n\}$ for all sufficiently large $n$. Since rounding changes each value by at most $1$, we have
    \begin{equation*}
        \frac{m_n^- - n/2}{\sqrt{n}} \longrightarrow c_\rho - \varepsilon,
        \qquad
        \frac{m_n^+ - n/2}{\sqrt{n}} \longrightarrow c_\rho + \varepsilon.
    \end{equation*}
    Lemma~\ref{lem:critical-window-rank} therefore gives
    \begin{equation*}
        K_{m_n^-}^{n}(\rho) > 0 \qquad \text{and} \qquad K_{m_n^+}^{n}(\rho) < 0
    \end{equation*}
    for all sufficiently large $n$. By Lemmas~\ref{lem:reduce-to-det-sign} and \ref{lem:determinant-recurrence}, these inequalities imply $f_{n,m_n^- - 1}^* > 0$ and $f_{n,m_n^+ - 1}^* < 0$. Since $f_{n,t}^*$ is nonincreasing in $t$ and $r_n$ is the first level at which this minimum is negative, we obtain
    \begin{equation*}
        m_n^- \leq r_n \leq m_n^+ - 1.
    \end{equation*}
    Consequently,
    \begin{equation*}
        c_\rho - \varepsilon \leq \liminf_{n \to \infty} \frac{r_n - n/2}{\sqrt{n}} \leq \limsup_{n \to \infty} \frac{r_n - n/2}{\sqrt{n}} \leq c_\rho + \varepsilon.
    \end{equation*}
    Letting $\varepsilon \downarrow 0$ proves \eqref{eq:rank-asymptotic}.
\end{proof}

Some numerical values are
\begin{center}
    \setlength{\tabcolsep}{5pt}
    \begin{tabular}{c|rrrrrrrr}
        $\rho$
        & $99/100$ & $9/10$ & $3/4$ & $3/5$
        & $1/2$ & $1/4$ & $1/10$ & $1/100$\\
        \hline
        $c_\rho$
        & $0.006286$ & $0.064705$ & $0.170732$ & $0.291092$
        & $0.382475$ & $0.682246$ & $0.991827$ & $1.528217$
    \end{tabular}
\end{center}

The dependence of $c_\rho$ on $\rho$ is illustrated in
Figure~\ref{fig:c-rho}. The curve is nonlinear over $(0,1)$ and exhibits different behavior at its two endpoints.
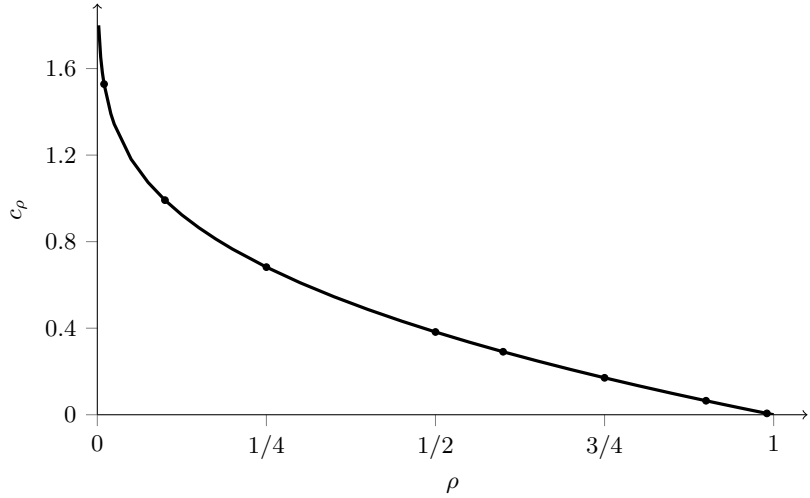
\begin{figure}[t]
    \centering
    \begin{tikzpicture}
        \begin{axis}[
            width=0.72\linewidth,
            height=0.46\linewidth,
            xmin=0,
            xmax=1.05,
            ymin=0,
            ymax=1.9,
            axis lines=left,
            axis line style={->},
            xlabel={$\rho$},
            ylabel={$c_\rho$},
            xtick={0,0.25,0.5,0.75,1},
            xticklabels={$0$,$1/4$,$1/2$,$3/4$,$1$},
            ytick={0,0.4,0.8,1.2,1.6},
            tick align=outside,
            tick label style={font=\small},
            label style={font=\small},
            clip=true
        ]
            \addplot[
                black,
                very thick
            ] coordinates {
                (0.002,1.799776)
                (0.005,1.651921)
                (0.0075,1.581002)
                (0.010,1.528217)
                (0.020,1.390964)
                (0.025,1.343216)
                (0.050,1.180982)
                (0.075,1.074179)
                (0.100,0.991827)
                (0.125,0.923611)
                (0.150,0.864730)
                (0.175,0.812532)
                (0.200,0.765385)
                (0.250,0.682246)
                (0.300,0.609896)
                (0.350,0.545322)
                (0.400,0.486659)
                (0.450,0.432668)
                (0.500,0.382475)
                (0.550,0.335444)
                (0.600,0.291092)
                (0.650,0.249045)
                (0.700,0.209004)
                (0.750,0.170732)
                (0.800,0.134029)
                (0.850,0.098733)
                (0.900,0.064705)
                (0.925,0.048129)
                (0.950,0.031828)
                (0.975,0.015788)
                (0.990,0.006286)
                (0.995,0.003138)
                (1.000,0.000000)
            };

            \addplot[
                only marks,
                mark=*,
                mark size=1.25pt,
                black
            ] coordinates {
                (0.010,1.528217)
                (0.100,0.991827)
                (0.250,0.682246)
                (0.500,0.382475)
                (0.600,0.291092)
                (0.750,0.170732)
                (0.900,0.064705)
                (0.990,0.006286)
            };
        \end{axis}
    \end{tikzpicture}
    \caption{The dependence of $c_\rho$ on $\rho$. The dots indicate the values reported in the preceding table. The curve was obtained numerically by solving $D_\rho(-2c_\rho) = 0$. It is truncated near $\rho = 0$, where $c_\rho \to \infty$, and uses the continuous extension
    $c_1 = 0$.}
    \label{fig:c-rho}
\end{figure}

The assumption that $\rho \in (0,1)$ is fixed is essential to the stated asymptotic conclusion. At the lower endpoint, $c_\rho \to \infty$ as $\rho \downarrow 0$. Indeed, fix $C > 0$. Dominated convergence in~\eqref{eq:kappa-definition-rank} gives
\begin{equation*}
    \kappa_\rho(C) \longrightarrow \int_0^\infty e^{-s^2/8}\cos(Cs)\,ds = \sqrt{2\pi}e^{-2C^2}>0.
\end{equation*}
The sign relation in Lemma~\ref{lem:critical-window-rank} therefore implies $c_\rho > C$ for all sufficiently small $\rho>0$. Since $C>0$ is arbitrary, it follows that $c_\rho \to \infty$. At the other endpoint, $c_\rho\to0$ as $\rho\uparrow1$. To see this, again fix $C>0$. Dominated convergence in~\eqref{eq:kappa-definition-rank} gives
\begin{equation*}
    \kappa_\rho(C) \longrightarrow \kappa_1(C) = -\int_0^\infty s e^{-s^2/8}\sin(Cs)\,ds = -4C\sqrt{2\pi}e^{-2C^2} < 0.
\end{equation*}
The same sign relation implies $0 < c_\rho < C$ whenever $\rho < 1$ is sufficiently close to one. Since $C > 0$ is arbitrary, the claim follows.

\medskip
{\bf Acknowledgments:} The authors used OpenAI’s GPT-5.6 Sol to assist in the derivations in Sections \ref{sec:rank-computation} and \ref{sec:asymptotic-analysis}. The authors independently verified all arguments and take full responsibility for the manuscript.

\bibliographystyle{plainnat}
\bibliography{references}

\end{document}